\documentclass[APA,LATO1COL]{WileyNJD-v2}

\articletype{Original Article}%

\received{10 May 2021}
\revised{2 June 2021}
\accepted{4 June 2021}

\definecolor{Base02}{HTML}{073642}
\input{tikz_header}

\hypersetup{
    colorlinks,
    linkcolor=Base02,
    citecolor=Base02,
    urlcolor=Base02
}

\usepackage[percent]{overpic}

\usepackage{amsmath,amsthm,amssymb}
\usepackage{graphicx}
\usepackage{enumitem}

\usepackage[noabbrev,nameinlink]{cleveref}

\newcommand{\abs}[2][]{#1\lvert #2 #1\rvert}
\newcommand{\ceil}[2][]{\lceil #2 \rceil}
\newcommand{\floor}[2][]{\lfloor #2 \rfloor}

\newcommand{\flx}{\ensuremath{\floor{x}}}
\newcommand{\clx}{\ensuremath{\ceil{x}}}
\newcommand{\flc}{\ensuremath{\floor{c}}}
\newcommand{\clc}{\ensuremath{\ceil{c}}}

\newcommand{\err}{\ensuremath{\operatorname{err}}}
\newcommand{\rd}{\ensuremath{\operatorname{rd}}}

\newcommand{\R}{\mathbb{R}}

\DeclareMathOperator{\VV}{\mathbb{V}}
\DeclareMathOperator{\EE}{\mathbb{E}}

\DeclareMathOperator{\PP}{\mathbb{P}}

\usepackage{bbm}

\renewcommand{\d}[1]{\ensuremath{\,\mathrm{d}#1}}

\usepackage{xspace}
\newcommand{\SR}{`stochastic round'\xspace}
\newcommand{\RtN}{`round to nearest'\xspace}

\begin{document}

\title{Non-asymptotic moment bounds for random variables rounded to non-uniformly spaced sets}

\author[1]{Tyler Chen}

\authormark{Tyler Chen}

\address{\orgdiv{Department of Applied Mathematics}, \orgname{University of Washington}, \orgaddress{\state{WA}, \country{USA}}}

\corres{\email{chentyl@uw.edu}}

\presentaddress{
Lewis Hall 201,
Box 353925,
University of Washington,
Seattle, WA 98195-3925}

\abstract[Summary]{
    We study the effects of rounding on the moments of random variables.
    Specifically, given a random variable \( X \) and its rounded counterpart \( \rd(X) \), we study \( \abs[\big]{ \!\EE[X^k] - \EE[\rd(X)^{k}] } \) for non-negative integer \( k \).
    We consider the case that the rounding function \( \rd : \mathbb{R}\to\mathbb{F} \) corresponds either to (i) rounding to the nearest point in some discrete set \( \mathbb{F} \) or (ii) rounding randomly to either the nearest larger or smaller point in this same set with probabilities proportional to the distances to these points.
    In both cases, we show, under reasonable assumptions on the density function of \( X \), how to compute a constant \( C \) such that \(  \abs[\big]{ \! \EE[X^k] - \EE[\rd(X)^{k}] } < C\epsilon^2 \), provided \( \abs{ \rd(x) - x } \leq \epsilon \: E(x) \), where \( E : \mathbb{R} \to \mathbb{R}_{\geq 0} \) is some fixed positive piecewise linear function. 
    Refined bounds for the absolute moments \( \EE\!\big[ |X^k-\rd(X)^{k}| \big] \) are also given.
} 

\jnlcitation{\cname{%
\author{T. Chen}} (\cyear{2021}), 
\ctitle{Non-asymptotic moment bounds for random variables rounded to non-uniformly spaced sets}, \cjournal{Stat}, \cvol{??}.}

\maketitle


\section{Introduction}

Rounded data is ubiquitous; measurements of length may be given in terms of the distance between consecutive markings on a ruler, weight may be measured to the nearest kilogram, age may be rounded to the nearest year, real numbers may be represented as floating point numbers, etc.
While it is often assumed that rounding error is small in comparison to other sources of error such sampling error, we are increasingly faced with settings in which there is a large amount of low precision data, making the task of understanding how the distribution of a random variable \( X \) and a rounded random variable \( \rd(X) \) relate increasingly important.

Famously, Sheppard studied the moments of \( \rd(X) \) obtained by rounding \( X \) to a set of uniform spacing \( h \).
In particular, he showed that, under suitable conditions on the density of \( X \), \( \EE[\rd(X)] \approx \EE[X] \) and \( \VV[\rd(X)] \approx \VV[X] + h^2/12 \) \cite{sheppard_97}.
The setting of rounding to a uniformly spaced set has remained of interest, with more recent work focusing on providing estimates and bounds under weaker conditions or for more general cases \cite{hall_82,tricker_90,wilrich_05,vardeman_05,janson_06,bai_zheng_zhang_hu_09,schneeweiss_komlos_ahmad_10,ushakov_ushakov_17}.
However, such past work makes critical use of the uniform spacing between points and is therefore neither applicable nor easily generalizable to the analysis of rounding to non-uniformly spaced sets.
The idea to consider a rounding random variable rounded to finite precision has been considered \cite{monahan_85}, although the statistical properties were never studied in detail.

In this paper, we show how to obtain bounds on the moments of \( \rd(X) \) for a wide range of rounding modes.
The techniques we develop for our analysis differ significantly from past work in that they make very limited assumptions about the set to which we are rounding; we simply require that the distance between consecutive points is bounded locally.
As a result, our techniques are applicable not only to the analysis of floating point number systems, but also to sets with irregular inter-point spacing such as those which might arise from sensor arrays.

\subsection{Related work}

The study of rounding random variables is closely related to the study of histograms.
Much of the theory on histograms focuses the quality of the histogram density (piecewise constant density function with mass of each bin equal to the mass of the underlying random variable over the given bin) \cite{freedman_diaconis_81,chaudhuri_motwani_narasayya_98,knuth_19}.
This differs slightly from the rounded random variable \( \rd(X) \) studied in this paper, which is a discrete random variable supported at the midpoints of the histogram bins, although the information contained in the two approaches is identical.
However, as with the study of rounding random variables, most analysis of histograms for density estimation study the case of uniformly spaced histogram bins.

Finally, we remark on several areas which are broadly related to this paper, and may be of interest to readers.
First, there is a great deal of work on the study of distribution quantization, which seeks to find a discrete random varibale to represent a continuous one \cite{graf_luschgy_07}.
Second, we contrast our work with rounding error analysis which makes statistical assumptions about the rounding errors incurred by a numerical algorithm \cite{wilkinson_63,higham_mary_19,connolly_higham_mary_21}.
We study the effect of deterministic perturbations to random variables; in fact, many of the results in this paper are derived directly from the deterministic structure of rounding errors.

\section{Setup}

\subsection{Finite precision number systems and rounding functions}

Let \( \mathbb{F} \subset \mathbb{R} \) be a discrete set on which the rounded random variable will be supported, and for notational convenience define \( \clx := \min\{ z \in \mathbb{F} :  z \geq x \} \) and \( \flx := \max\{ z \in \mathbb{F} : z \leq x \} \).
To ensure these quantities are well defined, we will assume that \( \inf \{ |x-y| : x,y \in \mathbb{F}, x\neq y \} > 0 \).

\label{sec:finite_precision}
\begin{figure}[ht]
\centering
        \begin{tikzpicture}

\def\x{10}
\def\y{\e*\x}

\def\a{11}
\def\b{12.5}
\def\e{(\b-\a)/(\a+\b)}
\def\ee{(1+\e)/(1-\e)}

\def\hh{.5}

\draw[Base02,line width=1,cap=round]
    (\a,0) -- ({(\a+\b)/2},{(\a-\b)/2});

\draw[Base02,line width=1,cap=round]
    ({(\a+\b)/2},{(\b-\a)/2}) -- (\b,0);

\path[Base02,line width=1,cap=round]
    (\a,{\b-\a}) -- (\b,{\a-\b});


\draw[<->] (\a-\hh,0) -- (\b+\hh,0); 

\def\h{.05}
\draw[line width=1,cap=round] (\a,-\h) -- (\a,\h);

\draw[dotted,cap=round] 
    (\a,{(\a-\b)}) -- (\a,{(\b-\a)});
\draw[dotted,cap=round] 
    (\b,{(\a-\b)}) -- (\b,{(\b-\a)});
 
\draw[line width=1,cap=round] (\b,-\h) -- (\b,\h);
    \node[below] at (\b,{\a-\b}) {\(\lceil x \rceil\)};
    \node[below] at (\a,{\a-\b}) {\(\lfloor x \rfloor\)};

\end{tikzpicture}
         \hspace{5em}
         \begin{tikzpicture}

\def\x{10}
\def\y{\e*\x}

\def\a{11}
\def\b{12.5}
\def\e{(\b-\a)/(\a+\b)}
\def\ee{(1+\e)/(1-\e)}

\def\hh{.5}

\path[left color=Base02, right color=white,shading path={draw=transparent!0,line width=1,cap=round}]
    (\a,0) -- (\b,{-(\b-\a)});

\path[right color=Base02, left color=white,shading path={draw=transparent!0,line width=1,cap=round}]
    (\a,{\b-\a}) -- (\b,0);


\draw[<->] (\a-\hh,0) -- (\b+\hh,0); 

\def\h{.05}
\draw[line width=1,cap=round] (\a,-\h) -- (\a,\h);
\draw[line width=1,cap=round] (\b,-\h) -- (\b,\h);

\draw[dotted,cap=round] 
    (\a,{(\a-\b)}) -- (\a,{(\b-\a)});
\draw[dotted,cap=round] 
    (\b,{(\a-\b)}) -- (\b,{(\b-\a)});
 
\draw[line width=1,cap=round] (\b,-\h) -- (\b,\h);
    \node[below] at (\b,{\a-\b}) {\(\lceil x \rceil\)};
    \node[below] at (\a,{\a-\b}) {\(\lfloor x \rfloor\)};

\end{tikzpicture}

    \caption{Error \( \err(x) := \rd(x) - x \) for selected rounding functions.
    \emph{Left}: \RtN, 
    \emph{Right}: \SR (darker colors represent higher probability).}
    \label{fig:rounding_functions}
\end{figure}
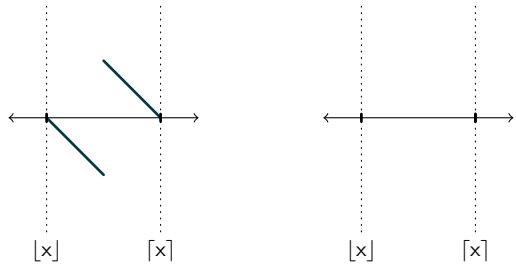

Given \( \mathbb{F} \), we consider two rounding functions \( \rd : \mathbb{R} \to \mathbb{F} \) defined by:
\hypertarget{rounding_schemes}{} 
\begin{align*}
    \text{round to nearest}
    &&
    \rd(x) 
    &:= \begin{cases}
        \flx, & x \leq \frac{1}{2}(\flx +  \clx )\\ 
        \clx, & x > \frac{1}{2}(\flx +  \clx )
    \end{cases}
    \\
    \text{stochastic rounding}
    &&
    \rd(x) 
    &:= \begin{cases}
        \flx, & \text{w.p. } 1 - (x-\flx)/(\clx - \flx ) \\ 
        \clx, & \text{w.p. } (x-\flx)/(\clx - \flx )
    \end{cases}
\end{align*}
The first is the standard \RtN scheme, which minimizes the distance between \( X \) and a random variable supported on \( \mathbb{F} \) in many metrics; e.g. ``earth mover'' distance, \( L^p \) norm, etc. The second is a randomized scheme which has gained popularity in recent years, particularly in machine learning \cite{gupta_agrawal_gopalakrishnan_narayanan_15,alistarh_grubic_li_tomioka_vojnovic_17,connolly_higham_mary_21}.
These schemes are illustrated in \cref{fig:rounding_functions}.

Once \( (\mathbb{F}, \rd) \) has been specified, we can consider how performing operations in the finite precision number system compare to performing the operations exactly.
For notational convenience, we define the error function,
\begin{align*}
    \err(x) := \rd(x) - x
\end{align*}
which tells us both the size and direction of rounding errors.
In principle, we could use this explicitly to compute quantities such as \( \EE[\rd(X)] \), but it would be exceedingly tedious to perform a separate analysis for every finite precision number system \( \mathbb{F} \).
As such, as is common in numerical analysis, we will use the assumption that,
\begin{align*}
    \abs{\err(x)} = \abs{ \rd(x) - x } \leq \epsilon \: E(x)
\end{align*}
for some fixed function \( E : \mathbb{R} \to \mathbb{R}_{\geq 0} \).
If \( E(x) = \abs{x} \) this bound is the standard bound for rounding to floating point number systems, and if \( E(x) = 1 \) this bound is the standard bound for fixed point systems; see for instance \cite{higham_02}.

Note that for a given set of numbers \( \mathbb{F} \) and non-negative function \( E \), the value of \( \epsilon \) required for \SR to satisfy \( \abs{ \err(x) }\leq \epsilon E(x) \) is roughly twice that of \RtN.
This is visible in \cref{fig:rounding_functions} and is at the core of the tradeoffs between the two approaches.

\subsection{Stochastic rounding}

When using stochastic rounding, \( \rd(X) \) is a random variable depending on both the randomness in \( X \) and in the rounding function \( \rd \).
We will assume that every time \( X \) is sampled, \( \rd \) is sampled according to its definition, \emph{independently of any past samples}; i.e. even if different samples of \( X \) take the same value, they might be rounded differently (although, in this paper we are concerned only with continuous random variables where the probability of two identical samples is zero).

We often require the expectation of \( \err(x)^k \) and \( \abs{\err(x)}^k \) taken over the randomness in the rounding function, so for convenience we give the following lemma.
\begin{lemma}
    \label{thm:stoch_err_ex}
    If \( \rd:\R\to\mathbb{F} \) is \SR, 
    \begin{align*}
        \EE_{\rd}\!\left[ \err(x)^k \right]
        &= (\flx - x )^k \left( 1- \frac{x-\flx}{\clx - \flx} \right)
        + (\clx - x )^k \left( \frac{x-\flx}{\clx - \flx} \right) 
        \\
        \EE_{\rd}\!\left[  \abs{ \err(x) } ^k \right]
        &=  ( x - \flx)^k \left( 1- \frac{x-\flx}{\clx - \flx} \right)
        +  (\clx - x)^k \left( \frac{x-\flx}{\clx - \flx} \right).
    \end{align*}
\end{lemma}
In particular, note that \( \EE_{\rd}[\err(x)] = 0 \); i.e. the rounding scheme is unbiased.

\subsection{Basic bounds}
\label{sec:basic_bounds}

\begin{theorem}
    \label{thm:basic}
    Suppose \( \EE[|X|^{k}] < \infty \) for some integer \( k > 0 \) and that \( E : \mathbb{R} \to \mathbb{R}_{\geq 0} \) is such that for some \( D \geq0 \) and sufficiently large \( |x| \), \( |E(x)| \leq D |x|  \).
    Then, there exists a constant \( C > 0 \) such that, for all \( \epsilon \in(0,1) \) and \( (\mathbb{F},\rd) \) where \( \rd: \R\to\mathbb{F} \) corresponds to \RtN or \SR and satisfies \( \abs{\rd(x) - x} < \epsilon \: E(x) \),
    \begin{align*}
        \abs[\Big]{ \EE\!\big[X^{k}\big] - \EE\!\big[\rd(X)^{k}\big] } 
        \leq \EE\!\left[\abs[\big]{X^{k} - \rd(X)^{k} }\right] 
        < C \epsilon.
    \end{align*}
   
\end{theorem}

\begin{proof}
By assumption \( \epsilon\in(0,1) \) so \( \epsilon^j \leq \epsilon \).
Moreover, because \( E \) has at most linear growth at infinity and the \( k \)-th absolute moment of \( X \) exists, each of the expectations \(  \EE[ \abs{ E(x)^{j} X^{k-j} } ] \) are finite. 
Thus for any \( j=1,\ldots, k \),
\begin{align*}
    \EE\!\left[ \abs[\big]{ \err(X)^{j} X^{k-j} } \right]
    \leq \EE\!\left[ \abs[\big]{ E(x)^{j} X^{k-j} } \right] \cdot \epsilon^{j}
    \leq \EE\!\left[ \abs[\big]{ E(x)^{j} X^{k-j} } \right] \cdot \epsilon 
    < \infty
\end{align*}

\noindent
Using that \( \rd(X) = X+ \err(X) \) we expand
\begin{align*}
    \rd(X)^{k} 
    =  X^{k} + \sum_{j=1}^{k} \binom{k}{j} X^{k-j} \err(X)^{j}.
\end{align*}
Then, applying the triangle inequality,
\begin{align*}
    \abs[\Big]{ \EE\!\big[X^{k}\big] - \EE\!\big[\rd(X)^{k}\big] } 
    \leq \left[ \sum_{j=1}^{k} \binom{k}{j} \EE\!\left[ \abs[\big]{ X^{k-j} \err(X) } \right]  \right] 
    \leq \left[ \sum_{j=1}^{k} \binom{k}{j} \EE\!\left[ \abs[\big]{ E(x)^jX^{k-j} } \right]\right] \cdot \epsilon.   
    \tag*{\qedhere}
\end{align*}
\end{proof}

\begin{table}[b]\centering
    \setlength{\tabcolsep}{16pt}
    \begin{tabular}{rccc}
        \toprule
        rounding scheme & \( d_{\rd}(k) \) & \( e_{\rd}(k) \) & \( f_{\rd}(k) \) \\
        \midrule
        \RtN & \( (k+1)^{-1} \) & \( (k+1)^{-1} \) & \( 2 (k+1)^{-1} \) \\
        \SR & \( \displaystyle (1 - (k+3) 2^{-(k+1)})(k^2+3k+2)^{-1} \) & \( \displaystyle 2(k^2+3k+2)^{-1} \) & \( \displaystyle 2(k^2+3k+2)^{-1} \)\\ 
        \bottomrule
    \end{tabular}
    \caption{
        Constants for different rounding schemes.
    }
    \label{table:constants}
\end{table}

\section{Higher order moment bounds}

If \( \rd \) corresponds to rounding with either \RtN or \SR, we expect cancellation in many cases. 
This is illustrated in \cref{fig:error_odd}, which depicts an error function corresponding to the \RtN scheme.
The key observation is that the integral of the error function of any finite interval is \emph{much} smaller than the integral of the corresponding bound.

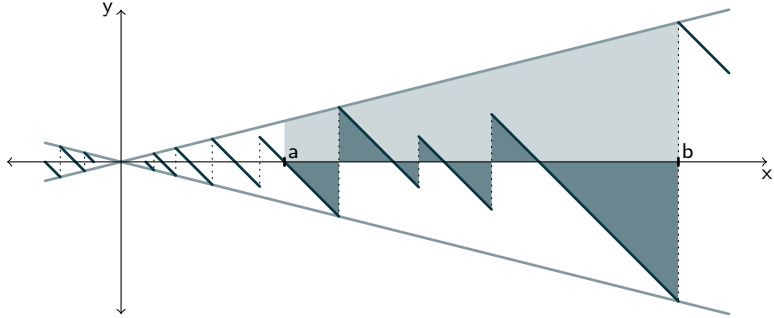
\begin{figure}[htb]\centering
    \begin{tikzpicture}

\def\x{8}
\def\e{1/4}
\def\y{\e*\x}
\def\ee{(1+\e)/(1-\e)}

\def\hh{0.5}

\def\a{2.15}
\def\b{{5.5/(1-\e)}}

\path[fill=Base02!20] (\a,\a*\e) -- (\b,{(\b)*\e}) -- (\b,0) -- (\a,0) -- cycle;
\draw[line width=1,Base02!50,cap=round] (\x,-\y) -- (0,0) -- (\x,\y);
\draw[line width=1,Base02!50,cap=round] (-2*\hh,{2*\hh*\e}) -- (0,0) -- (-2*\hh,{-2*\hh*\e});


\newcommand\drawError[2]{
    \begin{scope}
        \clip (\a,{(\a)*\e}) -- (\b,{(\b)*\e}) -- (\b,{(\b)*-\e}) -- (\a,{(\a)*-\e}) -- cycle;
        \path[fill=Base02!60] ({#1},0) -- ({(#1+#2)/2},{(#1-(#2))/2}) -- ({(#1+#2)/2},{-(#1-(#2))/2}) -- ({#2},0);
    \end{scope}
    
    \draw[dotted,cap=round] ({(#1+#2)/2},{-(#1-(#2))/2}) -- ({(#1+#2)/2},{(#1-(#2))/2});
    \draw[line width=1,Base02,cap=round] ({#1},0) -- ({(#1+#2)/2},{(#1-(#2))/2});
    \draw[line width=1,Base02,cap=round] ({(#1+#2)/2},{-(#1-(#2))/2}) -- ({min(#2,\x)},{0+(\x<#2)*(-(#1-(#2))/2 - (\x - (#1+#2)/2) )});
}

\def\lastx{1.5/(\ee)^3}
\foreach \xx[remember=\xx as \lastx] in {1.5/(\ee)^2,1.5/(\ee),1.5,\a,\a*\ee,4.25,5.5,5.5*\ee} {
    \drawError{\lastx}{\xx}
}

\def\lastx{-2*\hh/(\ee)^2}
\foreach \xx[remember=\xx as \lastx] in {-2*\hh/(\ee),-2*\hh} {
    \drawError{\lastx}{\xx}
}

\def\h{.05}
\draw[line width=1,cap=round] (\a,-\h) -- (\a,\h);
\draw[line width=1,cap=round] (\b,-\h) -- (\b,\h);
\node[above] at ({(\b)+.12},\h-.1) {\(b\)};
    \node[above] at ({(\a)+.12},\h-.1) {\(a\)};

\draw[<->] (0,-\y) -- (0,\y) node [left] {\(y\)};
\draw[<->] (-3*\hh,0) -- (\x+\hh,0) node [below] {\(x\)};

\end{tikzpicture}
    \caption{%
        The contribution of the integral of the \RtN error function over the interval \( [a,b] \) is at most the area of the rightmost darkly shaded triangle: \( [E(b)^2/2] \cdot \epsilon^2 \).
        This is in contrast to the lightly shaded area which is of size \( [\int_{a}^{b} E(x) \d{x}] \cdot \epsilon \).
        \emph{Legend}: 
        \( \err(x) \) ({\protect\raisebox{-.2mm}{\protect\includegraphics[scale=1]{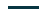}}}),
        \( E(x) \cdot \epsilon \) ({\protect\raisebox{-.2mm}{\protect\includegraphics[scale=1]{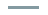}}}).
    }
    \label{fig:error_odd}
\end{figure}

\begin{lemma}
\label{thm:err_endpoint}
    Let \( c \in \mathbb{R} \) and \( c' \in \{ \flc,\clc \} \) so that \( |c-c'| \) is minimal.
    Then, for any odd integer \( k > 0 \), with \( d_{\rd}(k) \) as in \cref{table:constants},
    \begin{enumerate}[label=(\roman*),resume]
    \item 
    if \( \rd:\R\to\mathbb{F} \) is \RtN, 
    \begin{align*}
        \left| \int_{c}^{c'} \err(x)^k \d{x} \right|  
        \leq \left[ d_{\rd}(k) E(c)^{k+1} \right] \cdot \epsilon^{k+1}.
    \end{align*}
    \item 
        if \( \rd:\R\to\mathbb{F} \) is \SR, the above holds with \( \err(x)^k \) replaced by \( \EE_{\rd}[\err(x)^k] \).
    \end{enumerate}
\end{lemma}

\begin{proof}
    For \RtN \( c' = \rd(c) \). Note that \( \err(x) = \rd(c) - x \) does not change signs on \( [\min(c,c'),\max(c,c')] \).
    Thus, using the assumption \(  \abs{c-c'} \leq \epsilon \: E(c)  \),
    \begin{align*}
        \abs[\Bigg]{ \int_{c}^{c'} (\rd(c)-x)^k \d{x} }
        = \int_{\min(c,c')}^{\max(c,c')} |c'-x|^k \d{x} 
        = \frac{|c'-c|^{k+1}}{k+1}
        \leq \left[ \frac{1}{k+1} E(c)^{k+1 }\right] \cdot \epsilon^{k+1}.
    \end{align*}
    
    \noindent
    For \SR, \( \EE_{\rd}[\err(x)^k ] \) also does not change signs on \( [\min(c,c'),\max(c,c')] \).
    Moreover, it is symmetric about \( \bar{c} = (\flc+\clc)/2 \) on \( [\flc,\clc] \).
    Using these facts and that for \SR \( \clc-\flc \leq \epsilon \: E(c) \),
    \begin{align*}
        \abs[\Bigg]{ \int_{c}^{c'} \EE[\err(x)^k] \d{x} }
        \leq \int_{\flc}^{\bar{c}} \abs[\big]{\! \EE[\err(x)^k] } \d{x} 
        = \frac{1-(k+3)2^{-(k+1)}}{k^2+3k+2} (\clc-\flc)^{k+1}
        \leq \left[ \frac{1-(k+3)2^{-(k+1)}}{k^2+3k+2} E(c)^{k+1} \right] \cdot \epsilon^{k+1}.
        \tag*{\qedhere}
    \end{align*}
\end{proof}

From \cref{fig:error_odd}, it is clear that the contribution to integrals of \( \err(x) \) are due to the endpoints.
This is stated precisely in the following result, which is essentially a corollary of \cref{thm:err_endpoint}.
\begin{lemma}
    \label{thm:err_int_F}
    Let \( a, b \in \mathbb{R} \) with \( a<b \).
    Then, for any odd integer \( k > 0 \), with \( d_{\rd}(k) \) as in \cref{table:constants},
    \begin{enumerate}[label=(\roman*),resume]
    \item 
    if \( \rd:\R\to\mathbb{F} \) is \RtN, 
    \begin{align*}
        \abs[\Bigg]{ \int_{a}^{b} \err(x)^k \d{x} }
        \leq \left[ d_{\rd}(k)\max\left\{ E(a)^{k+1}, E(b)^{k+1} \right\} \right] \cdot \epsilon^{k+1}.
    \end{align*}
    \item 
        if \( \rd:\R\to\mathbb{F} \) is \SR, the above holds with \( \err(x)^k \) replaced by \( \EE_{\rd}[\err(x)^k] \).
    \end{enumerate}
\end{lemma}

\begin{proof}
We prove the \RtN case, but the same proof, with \( \err(x)^k \) replaced by \( \EE_{\rd}[\err(x)^k] \) holds for \SR.
Let \( a' \in \{ \lfloor a \rfloor, \lceil a \rceil\} \) and \( b' \in \{ \lfloor b \rfloor, \lceil b \rceil\} \) so that \( |a-a'| \) and \( |b-b'| \) are minimal.
    If \( k \) is odd then for any \( c \), by symmetry, the integral of \( \err(x)^k \) over \( [\flc,\clc] \) is zero.
Thus, inductively,
\begin{align*}
    \int_{a'}^{b'}\err(x)^k\d{x} = 0.
\end{align*}
Applying \cref{thm:err_endpoint} to the endpoints we have
\begin{align*}
    \abs[\Bigg]{ \int_{a}^{b} \err(x)^k \d{x} } 
    &= \abs[\Bigg]{ \int_{a'}^{b'} \err(x)^k \d{x} 
    + \int_{a}^{a'} \err(x)^k \d{x}
    + \int_{b'}^{b} \err(x)^k \d{x} }
    \\&\hspace{8.3em}= \abs[\Bigg]{ 
    \int_{a}^{a'} \err(x)^k \d{x}
    - \int_{b}^{b'} \err(x)^k \d{x} }
    \leq \left[ d_{\rd}(k)  \max \left\{ E(a)^{k+1} ,E(b)^{k+1} \right\} \right]\cdot \epsilon^{k+1}. 
    \tag*{\qedhere}
\end{align*}
\end{proof}

Next, we argue that this same higher order bound carries over to integrals of nice functions against \( \err(x)^k \).
First, however, we recall a basic property the Lebesgue--Stieltjes integral:
If \( g \) is absolutely integrable with respect to a measure \( \nu \), then,
\begin{align*}
    \mu(A) := \int_A g\d{\nu}
\end{align*}
is a (signed) measure and,
\begin{align*}
    \int_A fg \d{\nu}
    = \int_A f \d{\mu} .
\end{align*}


Let \( \mathcal{O} \subset 2^\mathbb{R} \) denote the set of all open subsets of \( \mathbb{R} \).
Recall that any open set \( A \in \mathcal{O} \setminus \{ \emptyset \} \) can be written \( A = \bigcup_{i=1}^{K} (a_i,b_i) \) where \( (a_i,b_i) \) are pairwise disjoint and \( K \in \mathbb{Z}_{>0} \cup \{\infty\} \).
Using this notation, we have the following lemma:
\begin{lemma}
\label{thm:prod_bound}
Let \( f: \mathbb{R} \to \mathbb{R}_{\geq 0} \) be lower semi-continuous and \( g:\mathbb{R} \to \mathbb{R} \) integrable. 
Suppose that \( fg \) is absolutely integrable and that there exists a function \( G : \mathbb{R} \times \mathbb{R} \to \mathbb{R}_{\geq 0} \) such that for any \( a,b\in\mathbb{R} \),
\begin{align*}
    \int_{a}^{b} g(x) \d{x}
    \leq G(a,b)
\end{align*}
Extend \( G : \mathbb{R} \times \mathbb{R} \to \mathbb{R}_{\geq 0} \) to a function \( \mu : \mathcal{O} \to \mathbb{R}_{\geq 0} \cup \{ \infty \} \) on open sets by \( \mu(\emptyset) = 0 \) and,
\begin{align*}
    \mu(A)
    = \mu\left( \bigcup_{i=1}^{K} (a_i,b_i) \right)
    = \sum_{i=1}^{K} G(a_i,b_i)
    ,&& \forall A \in \mathcal{O} \setminus \{ \emptyset \} 
\end{align*}
Then, 
\begin{align*}
    \int_{\mathbb{R}} f(x) g(x) \d{x} 
    \leq \int_{0}^{\infty} \mu ( \{ x : f(x) > u\} ) \d{u}.
\end{align*}
\end{lemma}

\begin{proof}
Define \( \nu(A) := \int_{A} g(x) \d{x} \) and observe that by definition \( \nu(A) \leq \mu(A) \) for all \( A  \in \mathcal{O} \).
Then, by definition of Lebesgue--Stieltjes integral,
\begin{align*}
    \int_{\mathbb{R}} f(x) g(x) \d{x} 
    &= \int_{0}^{\infty} \nu(\{ x : f(x) > u\}) \d{u}
    \leq \int_{0}^{\infty} \mu(\{ x : f(x) > u\}) \d{u}.
    \tag*{\qedhere}
\end{align*}
\end{proof}

This lemma allows us to provide lower bounds on the integral of \( fg \) given lower bounds on the integral of \( g \) over \( [a,b] \) (simply replace \( g \) with \( -g \) in the above statement).
This bound can therefore be naturally extended to apply to any function which have negative outputs by decomposing \( f  = f^+ + f^- \) where \( f^+, -f^- \geq 0 \) provided both \( f^+ \) and \( -f^- \) are lower semi-continuous.
Moreover, if \( G(a,b) \) is of the form \( G(a,b) = \int_a^b h(x) \d{x} \), then
\begin{align*}
    \int_{0}^{\infty} \mu(\{x : f(x) > u \} \d{u}
    = \int_{\mathbb{R}} f(x) h(x) \d{x}.
\end{align*}

How tight \cref{thm:prod_bound} is depends on how tight the bound for the integral of \( g \) is.
For instance, if the bound on \( g \) is equality then the proposition's bound is equality; in fact it is simply the Lebesgue--Stieltjes integral \( \int f(x) \d G(x) \), where \( G \) is an antiderivative of \( g \).
On the other hand, as in the case of our subsequent applications of this proposition, if the bound on \( g \) does not take into account some behavior of \( g \), then the bound will be more pessimistic as it must account for the worst case interaction between \( f \) and \( g \).

To facilitate the use of \cref{thm:prod_bound}, we introduce the following definition and Lemma.
\begin{definition}
A function \( f : \R \to \R \) is said to have \( K \) regions of local maxima if, for all \( t \in \R \), 
\begin{align*}
    \{ x : f(x) > t \} = \bigcup_{i=1}^{K} (a_i, b_i)
\end{align*}
where \( (a_i, b_i) \) are pairwise disjoint.
\end{definition}

\begin{lemma}
\label{thm:unimodal}
    Suppose \( f : (a,b) \to \mathbb{R}_{\geq 0} \) is bounded and has single region of local maxima, and that \( E(x) = mx+c \) for \( m \geq 0 \).
Let \( x^* \in \mathbb{R} \) be the largest point such that \( f \) is non-decreasing on \( (a,x^*) \) and non-increasing on \( (x^*,b) \).
Then, if \( k \) is odd,
\begin{enumerate}[label=(\roman*),nolistsep]
    \item If \( \rd:\R\to\mathbb{F} \) is \RtN,
        \begin{align*}
            &\left| \int f(x) \err(x)^{k} \d{x} \right|
            \leq \left[ m (k+1) d_{\rd}(k) \int_{x^*}^{b} ( mx+c)^k f(x) \d{x} \right] \cdot \epsilon^{k+1}.
        \end{align*}
    \item if \( \rd:\R\to\mathbb{F} \) is \SR, the above holds with \( \err(x)^k \) replaced by \( \EE_{\rd}[\err(x)^k] \).
\end{enumerate}
\end{lemma}

\begin{proof}
We prove the \RtN case, but the same proof, with \( \err(x)^k \) replaced by \( \EE_{\rd}[\err(x)^k] \), holds for \SR.
    First, we make several notational definitions.
    Define \( \hat{f}:(x^*,b) \to (0,f(x^*)) \) as the restriction of \( f \) to \( (x^*,b) \). 
    That is, for all \( x\in(x^*,b) \), \( \hat{f}(x) = f(x) \).
    Next, define, \( \hat{f}^{-1}: (0,f(x^*)) \to (x^*,b) \) by
    \begin{align*}
        \hat{f}^{-1}(u) := \sup\{ x : f(x) > u \} .
    \end{align*}
    Define also
    \begin{align*}
        G(a',b') :=  \left[ d_{\rd}(k) \max\{E(a')^{k+1}, E(b')^{k+1}\} \right] \epsilon^{k+1}
        =  \left[ d_{\rd}(k) E(b')^{k+1} \right] \epsilon^{k+1}.
    \end{align*}

\noindent
By assumption, \( f \) has a single local maxima (or connected region of local maxima) so
\begin{align*}
    \{ x : f(x) > u \}
    = \left( \inf \{ x : f(x) > u \}, \sup \{ x : f(x) > u \} \right).
\end{align*}
Then, in the notation of \cref{thm:prod_bound} with \( g(x) = \err(x)^k \), for \( u < f(x^*) \),
\begin{align*}
    \mu( \{ x : f(x) > u \} )
    &= G\left( \inf \{ x : f(x) > u \}, \sup \{ x : f(x) > u \} \right) 
    \\&\leq d_{\rd}(k) E(\sup\{ x : f(x) > u \})^{k+1}
    \\&= d_{\rd}(k) E(\hat{f}^{-1}(u))^{k+1}
\end{align*}
where the inequality follows from the fact \( E \) is non-decreasing.
Therefore, applying \cref{thm:prod_bound} and \cref{thm:err_int_F},
\begin{align*}
    \abs[\bigg]{ \int_a^b f(x) \err(x)^k \d{x} }
    &\leq \abs[\bigg]{ \int_0^\infty \mu( \{ x : f(x) > u \} ) \d{u} \: } 
    \leq \left[ d_{\rd}(k) \abs[\bigg]{ \int_0^{f(x^*)} E(\hat{f}^{-1}(u))^{k+1} \d{u} \: } \right] \cdot \epsilon^{k+1}.
\end{align*}
Since \( E \) and \( \hat{f}^{-1} \) are non-decreasing, \( u\mapsto E(\hat{f}^{-1}(u))^{k+1} \) is also non-decreasing.
    Thus, reversing the axis of integration and making a change of variables \( v = E(x)^{k+1} = (mx+c)^{k+1} \),
\begin{align*}
    \int_0^{f(x^*)} E(\hat{f}^{-1}(u))^{k+1} \d{u}
    &= \int_{E(x^*)^{k+1}}^{E(b)^{k+1}} \hat{f} \left( \frac{v^{1/(k+1)}-c}{m} \right) \d{v}
    = (k+1) m\int_{x^*}^{b} ( mx+c)^k \:\hat{f} (x) \d{x}.
    \tag*{\qedhere}
\end{align*}
\end{proof}

\begin{theorem}
    Suppose \( \EE[|X|^{k}] < \infty \) for some integer \( k > 0 \), that \( x\mapsto x^{\alpha-1} f_X(x) \) has finitely many regions of local maxima, and that \( E : \mathbb{R} \to \mathbb{R}_{\geq 0} \) is piecewise linear with a finite number of breakpoints.

    Then there exists a constant \( C > 0 \) such that, for all \( \epsilon \in(0,1) \) and \( (\mathbb{F},\rd) \) where \( \rd: \R\to\mathbb{F} \) corresponds to \RtN or \SR and satisfies \( \abs{\rd(x) - x} < \epsilon \: E(x) \),
    \begin{align*}
       \abs[\Big]{ \EE\!\big[X^{k}\big] - \EE\!\big[\rd(X)^{k}\big] } < C \epsilon^2.
    \end{align*}
   
\end{theorem}

\begin{proof}
As in the proof of \cref{thm:basic} we expand
\begin{align*}
    \EE\!\left[ \rd(X)^{k} \right]
    = \EE\!\left[X^{k}\right] + \sum_{j=1}^{k} \binom{k}{j} \EE\!\left[ X^{k-j} \err(X)^{j} \right]
\end{align*}
and note that each term in the sum is of size \( O(\epsilon^j) \).
It suffices to show that the \( j=1 \) term is actually \( O(\epsilon^2) \).

\noindent
Let \( K_1 \) be the number of regions of local maxima and \( K_2 \) the number of breakpoints in \( E \).
    Then we can partition \( (-\infty,\infty) \) into \( K \leq K_1+K_2 \) intervals \( (a_i, a_{i+1}) \) such that on each \( (a_i, a_{i+1}) \), \( x\mapsto x^{\alpha-1} f_X(x) \) has a single region of local maxima and \( E \) is piecewise linear.
Then, applying \cref{thm:unimodal} to each interval we see that for some constant \( C_1 > 0 \)
\begin{align*}
    \left| \int \err(x) x^{k-1} f_X(x) \d{x} \right|
    &\leq \sum_{i=1}^{K} \left| \int_{a_i}^{a_{i+1}} \err(x) x^{k-1} f_X(x) \d{x} \right|
    \leq C_1 \epsilon^2.
\end{align*}
    The result then follows by adding \( C_1 \) to the coefficients for the bounds for the \( j > 1 \) terms, similar to as in the proof of \cref{thm:basic}.
\end{proof}

\section{Refined absolute moment bounds}

Trivially, we may bound the integral of \( \abs{ \err(x) }^k \) by the integral of \( \epsilon^k \: E(x)^k \).
However, as suggested by \cref{fig:error_even}, if \( \rd \) corresponds to \RtN or \SR, integrals against \( \abs{ \err(x) }^k \) should be a constant fraction smaller than integrals against \( \epsilon^k \:E(x)^k \) due to the fact that the rounding function cannot always attain the size of the worst case error.
As in the proof of \cref{thm:err_int_F}, we bound the bulk of the contribution between two numbers in \( \mathbb{F} \), and then account for tiny contributions at the endpoints.

\begin{lemma}
\label{thm:err_int_abs_fp}
    Suppose \( E(x) = mx+b \) on \( [\flc,\clc] \).
    Then for any integer \( k > 0 \), with \( e_{\rd}(k) \) as in \cref{table:constants},
    \begin{enumerate}[label=(\roman*),resume]
    \item
    if \( \rd:\R\to\mathbb{F} \) is \RtN,
    \begin{align*}
        \int_{\flc}^{\clc} \abs{\err(x)}^k \d{x} 
        \leq \left[ e_{\rd}(k) \int_{\flc}^{\clc} E(x)^k \d{x} \right] \cdot \epsilon^k .
    \end{align*}
    \item 
        if \( \rd:\R\to\mathbb{F} \) is \SR, the above holds with \( \abs{ \err(x) }^k \) replaced by \( \EE_{\rd}[\abs{\err(x)}^k] \).
    \end{enumerate}
\end{lemma}

\begin{proof}
    We first prove the \RtN case.
    By direct computation, using that \( \err((\flc+\clc)/2) \leq \epsilon \: E((\flc+\clc)/2) \) followed by the fact that the tangent to \( E(x)^k \) at \( (\flc+\clc)/2 \) lies entirely below \( E(x) \),
    \begin{align*}
        \int_{\flc}^{\clc} \abs{\err(x)}^k \d{x}
        &= \frac{1}{k+1}\int_{\flc}^{\clc} \abs[\bigg]{ \err \left( \frac{\flc+\clc}{2} \right)  }^k \d{x}
        \leq \left[ \frac{1}{k+1} \int_{\flc}^{\clc} E \left( \frac{\flc+\clc}{2} \right)^k \d{x} \right] \cdot \epsilon^k
        \leq \left[ \frac{1}{k+1} \int_{\flc}^{\clc} E(x)^k \d{x} \right] \cdot \epsilon^k.
    \end{align*}
    
    \noindent
    For the \SR case, again by direct computation, followed by the fact that \( (\flc-\clc) \leq E(x) \), we find
    \begin{align*}
        \int_{\flc}^{\clc} \EE_{\rd}\!\big[\abs{\err(x)}^k\big] \d{x}
        = \frac{2}{k^2+3k+2} \int_{\flc}^{\clc} \left( \clc-\flc \right)^{k} \d{x}
        \leq \left[ \frac{2}{k^2+3k+2} \int_{\flc}^{\clc} E(x)^k \d{x} \right] \cdot \epsilon^k
   \tag*{\qedhere}
    \end{align*}
\end{proof}

\begin{figure}[ht]\centering
    \begin{tikzpicture}
\def\x{10}
\def\y{\e*\x}

\def\a{11}
\def\b{14}
\def\e{(\b-\a)/(\a+\b)}
\def\ee{(1+\e)/(1-\e)}

\def\hh{.65}

\path[fill=Base02!20,domain={\a}:{\b},smooth,variable=\x] 
    (\a,0) 
    -- plot ({\x},{(\x*\e)^2+\hh}) 
    -- (\b,0) 
    -- cycle;

\path[fill=Base02!60,domain=0:1,smooth,variable=\x] 
    plot ({\a*(1-\x)+\x*(\a+\b)/2},{(\a*(1-\x)+\x*(\a+\b)/2-\a)^2}) 
    -- plot ({(\a+\b)/2*(1-\x)+\x*\b},{((\a+\b)/2*(1-\x)+\x*\b-\b)^2}) 
    -- cycle;

\draw[line width=1,black!40,cap=round,TwoMarks={1.5*\hh cm}] (\a-3.5*\hh,{(\b-\a)/2)^2}) -- (\b+\hh,{(\b-\a)/2)^2}); 

\draw[line width=1,black!20,cap=round,TwoMarks={1.5*\hh cm}] (\a-3.5*\hh,{(\b-\a)/2)^2+\hh}) -- (\b+\hh,{(\b-\a)/2)^2+\hh}); 

\draw[Base02,line width=1,domain=0:1,smooth,variable=\x,cap=round]
    plot ({(\a-\hh)*(1-\x)+\x*(\a+\b)/2},{((\a-\hh)*(1-\x)+\x*(\a+\b)/2-\a)^2})
    plot ({(\a+\b)/2*(1-\x)+\x*(\b+\hh)},{((\a+\b)/2*(1-\x)+\x*(\b+\hh)-\b)^2});

\draw[Base02!50,line width=1,domain={\a-\hh}:{\b+\hh},smooth,variable=\x,cap=round] plot ({\x},{(\x*\e)^2+\hh});


\draw[<->] (\a-3*\hh,{-\hh/2}) -- (\a-3*\hh,{((\b+\hh)*\e)^2}) node[left] {\(y\)};
\draw[<->,TwoMarks={2*\hh cm}] (\a-4*\hh,0) -- (\b+2*\hh,0) node[below] {\(x\)};

\def\h{.05}
\draw[line width=1,cap=round] (\a,-\h) -- (\a,\h);
\draw[line width=1,cap=round] (\b,-\h) -- (\b,\h);
    \node[below] at (\b,\h-.06) {\(\lceil c \rceil\)};
    \node[below] at (\a,\h-.06) {\(\lfloor c \rfloor\)};

\draw[dotted,cap=round] (\a,0) -- (\a,{((\b-\a)/2)^2+1}); 
\draw[dotted,cap=round] (\b,0) -- (\b,{((\b-\a)/2)^2+1}); 

\draw[dotted,domain={\a-\hh}:{\b+\hh},smooth,variable=\x,cap=round] plot ({\x},{2*\e*\e*((\a+\b)/2)*(\x - (\a+\b)/2) + ((\b-\a)/2)^2+\hh});

\end{tikzpicture}

    \caption{%
        The contribution of the integral of the \( k \)-th power of absolute error function for \RtN over the interval \( [ \floor{c},  \ceil{c} ] \) is at most \( 1/(k+1) \) the area of the integral of the constant function \( (\ceil{c} - \floor{c})^k/2^k \) over this interval, which is itself smaller than the integral of \( \epsilon^k \abs{x}^k \) over this interval.
        Note that \( E(x) \) is always larger than the tangent at \( (\clc+\flc)/2 \) and that the area under this tangent is equal to that of the constant function through the tangent point.
        \emph{Legend}: 
        \( E(x)^k \cdot \epsilon^k \) ({\protect\raisebox{-.2mm}{\protect\includegraphics[scale=1]{blue1_solid.pdf}}}),
        \( \left( \frac{\flc+\clc}{2} \right)^k \cdot \epsilon^k \) ({\protect\raisebox{-.2mm}{\protect\includegraphics[scale=1]{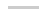}}}),
        \( \left( \frac{\flc-\clc}{2} \right)^k \) ({\protect\raisebox{-.2mm}{\protect\includegraphics[scale=1]{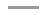}}}),
        \( \abs{ \err(x) }^k \) ({\protect\raisebox{-.2mm}{\protect\includegraphics[scale=1]{blue0_solid.pdf}}}).
    }
    \label{fig:error_even}
\end{figure}

\begin{lemma}
\label{thm:err_int_abs_fp_end}
    Suppose \( E(x) = mx+b \) on \( [\flc, \clc] \).
    Then for any integer \( k > 0 \), with \( f_{\rd}(k) \) as in \cref{table:constants},
    \begin{enumerate}[label=(\roman*),resume]
    \item 
    if \( \rd:\R\to\mathbb{F} \) is \RtN, for \( \beta = 1/(1-m\epsilon) \),
    \begin{align*}
        \int_{\flc}^{\clc} \abs{\err(x)}^k \d{x} 
        \leq \left[ f_{\rd}(k) (\beta E(c))^{k+1} \right] \cdot \epsilon^{k+1}.
    \end{align*}
    \item 
    if \( \rd:\R\to\mathbb{F} \) is \SR, the above holds with \( \abs{ \err(x) }^k \) replaced by \( \EE_{\rd}[\abs{ \err(x) }^k] \) and \( \beta=1 \).
    \end{enumerate}
\end{lemma}

\begin{proof}
   
    We first prove the \RtN case.
    Using that \( E(x) = mx+b \) we have
    \begin{align*}
        \frac{\clc-\flc}{2} 
        = \abs[\bigg]{ \err \left( \frac{\flc+\clc}{2} \right) } 
        \leq E\left( \frac{\flc+\clc}{2} \right) \cdot \epsilon
        \leq \left[ E(c) + m \left( \frac{\clc-\flc}{2} \right) \right] \cdot \epsilon.
    \end{align*}
    We therefore find that
    \begin{align*}
        \frac{\flc-\clc}{2} \leq \frac{\epsilon}{1-m \epsilon} E(c).
    \end{align*}
    Next, using that \( \err(x) \) is symmetric about \( (\flc+\clc)/2 \) on \( [\flc,\clc] \)
    \begin{align*}
        \int_{\flc}^{\clc} \abs{\err(x)}^k \d{x}
        = \frac{2}{k+1} \left( \frac{\clc-\flc}{2} \right)^{k+1}
        \leq \left[ \frac{2}{k+1} E(c)^{k+1} \right] \cdot \left( \frac{\epsilon}{1-L\epsilon} \right)^{k+1}.
    \end{align*}
    
    \noindent
    We now prove the \RtN case.
    Here we have that \( \flc-\clc \leq \epsilon \: E(c) \) so
    \begin{align*}
        \int_{\flc}^{\clc} \EE_{\rd}[ \abs{\err(x)}^k] \d{x}
        = \frac{2}{k^2+3k+2} \left( \clc-\flc \right)^{k+1}
        \leq \left[ \frac{2}{k^2+3k+2} E(c)^{k+1} \right] \cdot \epsilon^{k+1}.
        \tag*{\qedhere}
    \end{align*}
\end{proof}

Combining \cref{thm:err_int_abs_fp} with \cref{thm:err_endpoint} we obtain the following:
\begin{theorem}
    \label{thm:err_int_abs}
    Suppose that for all \( x \), \( E \) is piecewise linear on \( [\flx, \clx] \) with maximum slope \( m \).
    Then,
    \begin{enumerate}[label=(\roman*),resume]
    \item 
    if \( \rd:\R\to\mathbb{F} \) is \RtN, for \( \beta = 1/(1-m\epsilon) \),
    \begin{align*}
        \abs[\bigg]{ \int_{a}^{b} \abs{\err(x)}^k \d{x} }
        &\leq \left[ e_{\rd}(k) \int_{a}^{b} E(x)^k \d{x} \right] \cdot \epsilon^k
        + \left[ f_{\rd}(k)\left( (\beta E(a))^{k+1} + (\beta E(b))^{k+1} \right) \right] \cdot \epsilon^{k+1}
    \end{align*}
    \item 
        if \( \rd:\R\to\mathbb{F} \) is \SR, the above holds with \( \abs{ \err(x) }^k \) replaced by \( \EE_{\rd}[\abs{ \err(x) }^k] \) and \( \beta=1 \).
    \end{enumerate}
\end{theorem}

\begin{proof}
    Similar to the proof of \cref{thm:err_int_F} we have,
    \begin{align*}
        \int_{a}^{b} \err(x)^n \d{x}
        &= \int_{\ceil{a}}^{\floor{b}} \err(x)^n \d{x} + \int_{a}^{\ceil{a}} \err(x)^n \d{x}
        + \int_{\floor{b}}^{b} \err(x)^n \d{x}
    \end{align*}
    By \cref{thm:err_int_abs_fp},
    \begin{align*}
        \int_{\ceil{a}}^{\floor{b}} \err(x)^n \d{x} 
        &\leq \left[ e_{\rd}(k) \int_{\ceil{a}}^{\floor{b}} E(x)^k \d{x} \right] \cdot \epsilon^k
        \leq \left[ e_{\rd}(k) \int_{a}^{b} E(x)^k \d{x} \right] \cdot \epsilon^k
    \end{align*}
    The result follows by applying \cref{thm:err_endpoint} to the remaining terms, accounting for the sign of the integrand.
    \qedhere
\end{proof}

\subsubsection{Two sided bounds for uniform meshes}

When \( \mathbb{F} \) contains uniformly spaced points, then we can take \( E \) to be constant so \cref{thm:err_int_abs_fp} becomes equality.
This provides something akin to a non-asymptotic version of Sheppard's corrections.
\begin{theorem}
\label{thm:err_int_abs_unif}
Suppose that \( \mathbb{F} = \{ a + 2 \delta z : z\in\mathbb{Z} \} \),
Then
\begin{enumerate}[label=(\roman*),resume]
    \item 
    if \( \rd:\R\to\mathbb{F} \) is \RtN, \( \epsilon = \delta \)
    \begin{align*}
        \abs[\Bigg]{ \int_{a}^{b}  \abs{ \err(x) }^k \d{x}
        - \left[ e_{\rd}(k) \int_{a}^{b} 1 \d{x} \right] \cdot \epsilon^k }
        \leq \left[ 2 f_{\rd}(k) \right] \cdot \epsilon^{k+1}
    \end{align*}
    \item 
        if \( \rd:\R\to\mathbb{F} \) is \SR, \( \epsilon = 2 \delta \) and the above holds with \( \abs{\err(x)}^k \) replaced with \( \EE_{\rd}[\abs{\err(x)}^k] \).
    \end{enumerate}

\end{theorem}

\begin{proof}
    We first prove the \RtN case.
    By direct computation, for any \( c\in [a,b] \),
    \begin{align*}
        \int_{\flc}^{\clc} \abs{\err(x)}^k
        &= \frac{1}{k+1} \int_{\flc}^{\clc} \epsilon^k \d{x}.
    \end{align*}
    The result then follows by the same approach as \cref{thm:err_int_abs}.
    
    \noindent
    For the \SR case, as with past computations, we obtain the constant \( 2(k^2+3k+2)^{-1} \).
    The computation matches that of \cref{thm:err_int_abs_fp_end}.
    \qedhere
\end{proof}

\section{Example}

\label{ex:semicircle}
In this example, we illustrate how our techniques can be used to provide bounds on the mean, variance, and several other quantities corresponding to a rounded random variable.

Suppose \( X \) is distributed according to the semi-circle distribution with mean \( \mu \) and radius \( r \).
That is \( f_X(x) = \frac{2}{\pi r^2} \sqrt{r^2 - (x-\mu)^2} \) for \( -r \leq x-\mu \leq r \) and \( f_X(x) = 0 \) otherwise.
    We consider the effect of rounding to the set \( \{ a + 2\delta z:z\in\mathbb{Z} \} \) on the quantities,
\begin{align*}
    \abs{\EE[\rd(X)] - \EE[x] }
    ,&&
    \abs{\VV[\rd(X)] - \VV[x] }
    ,&&
    \abs{\EE[X\err(X)]}
    ,&&\text{and} &&
    \EE[\err(X)^2].
\end{align*}
For each value of \( \delta \), we consider many values of \( a \in [0,2\delta) \), taking the supremum over \( a \).
As seen in \cref{fig:error_semicircle}, even for a fixed value of \( \delta \), these quantities may vary drastically as \( a \) changes.
This illustrates the advantage of bounds which depend only on limited information about the set being rounded to.

We now bound the differences of the mean and variances based on increasing amounts of information:
\begin{enumerate}[label=(\alph*),nolistsep]
    \item[(a)] \( \abs{\rd(x) -x} \leq \delta \)
    \item[(b)] \( \rd:\R\to\mathbb{F} \) is \hyperlink{rounding_schemes}{`round to nearest'} 
    \item[(c)] \( \mathbb{F} \) has uniform spacing \( 2 \delta \).\
\end{enumerate}%

\begin{figure*}[t]\centering
    \begin{overpic}[width=\textwidth]{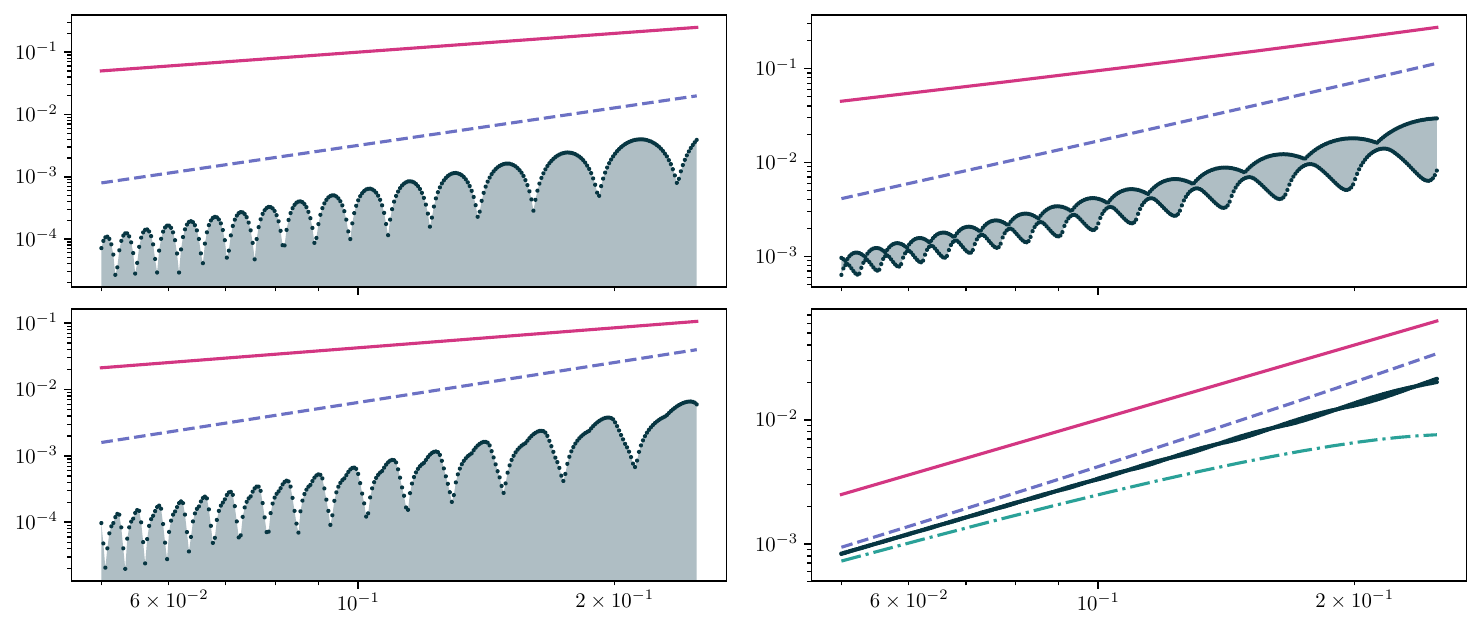}
        \put (6,39) {\( \abs{ \EE[\rd(X)] - \EE[X] } \)}
        \put (56,39) {\( \abs{ \VV[\rd(X)] - \VV[X] } \)}
        \put (6,19) {\( \abs{ \EE[X\err(X)] } \)}
        \put (56,19) {\( \abs{ \EE[\err(X)^2] } \)}
        \put (20,-.5) {additive error bound \(\delta\)}
        \put (70,-.5) {additive error bound \(\delta\)}
    \end{overpic}

    \caption{Error bounds when \( X \) has semicircle distribution and is rounded to a uniformly spaced set \( \mathbb{F} = \{ a + 2\delta z : z\in \mathbb{Z}\} \).
    \emph{Legend}:
    The true error, for all values of \( a\in[0,2\delta) \), is show as the shaded region,
    bound (a) ({\protect\raisebox{-.2mm}{\protect\includegraphics[scale=1]{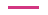}}}),
    bound (b) ({\protect\raisebox{-.2mm}{\protect\includegraphics[scale=1]{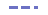}}}),
    bound (c) ({\protect\raisebox{-.2mm}{\protect\includegraphics[scale=1]{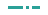}}}).
    }
    \label{fig:error_semicircle}
\end{figure*}

Using only that \( \abs{\err(x)} \leq \delta \), we find
\begin{align*}
    \abs{ \EE[ \err(X) ]} \leq \delta
    ,&&
    \abs{ \EE[ X\err(X) ]} \leq \EE[|X|] \cdot \delta
    \leq \left[ \frac{4r}{3\pi} \right] \cdot \delta
    ,&&
    \abs{ \EE[ \err(X)^2 ]} \leq \delta^2.
    \tag*{bound (a)}
\end{align*}

If we additionally know that \( \rd \) is \hyperlink{rounding_schemes}{`round to nearest'}, then we can improve our bounds to quadratic in \( \delta \).
In particular, using \cref{thm:unimodal},
\begin{align*}
    \abs{ \EE[ \err(X) ] }
    &= \abs[\bigg]{ \int_{\mathbb{R}} \err(x) f_X(x) \d{x} }
    \leq \left[ \frac{1}{2} \left( \sup_x f_X(x) \right) \right] \cdot \delta^2
    = \left[ \frac{1}{\pi r} \right] \cdot \delta^2.
    \tag*{bound (b)}
\end{align*}
Likewise, but noting that \( x f_X(x) \) changes sign (hence the factor of 2),
\begin{align*}
     \abs{ \EE[ X \err(X) ] }
     &= \abs[\bigg]{ \int_{\mathbb{R}} x \err(x) f_X(x) \d{x} }
     \leq \left[ \frac{2}{2} \left( \sup_x |x f_X(x)| \right) \right] \cdot \delta^2
     \leq \left[ \frac{1}{\pi} \right] \cdot \delta^2.
     \tag*{bound (b)}
\end{align*}
Using only \( \abs{\err(x)} \leq \delta \), we already have that \( \EE[ \err(X) ]^2 = \mathcal{O}(\delta^2) \).
    However, using \cref{thm:unimodal} we can improve the constant to,
    \begin{align*}
        \abs{ \EE[ \err(X)^2 ] }
        &= \abs[\bigg]{ \int_{\mathbb{R}} \err(x)^2 f_X(x) \d{x} }
        \leq \left[ \frac{1}{3} \abs[\bigg]{ \int_{}^{} f_X(x) \d{x} } \right] \cdot \delta^2 
        + \left[ \frac{4}{3} \sup_x f_X(x) \right] \cdot \delta^3
        \leq \left[ \frac{1}{3} \right] \cdot \delta^2 + \left[ \frac{8}{3 \pi r} \right] \cdot \delta^3.
         \tag*{bound (b)}
    \end{align*}

    Finally, if our mesh is uniform we can use \cref{thm:err_int_abs_unif} to provide a two sided bound. 
    We note the upper bound is the same as our previous upper bound, but we can now provide an corresponding lower bound.
    \begin{align*}
        \abs[\bigg]{ \EE[\err(X)^2] - \frac{1}{3} \cdot \delta^2 } 
        &= \abs[\bigg]{ \int_{\mathbb{R}} \err(X)^2 f_X(x) \d{x} - \frac{1}{3} \cdot \delta^2 }
        \leq \left[ \frac{4}{3} \left( \sup_x f_X(x) \right) \right] \cdot \delta^3
        = \left[ \frac{8}{3 \pi r} \right] \cdot \delta^3.
        \tag*{bound (c)}
    \end{align*}

    These bounds, are shown in \cref{fig:error_semicircle}.
    In general we have that,
    \begin{align*}
        \VV[\rd(X)] - \VV[X] = \VV[\err(X)] + 2 \textrm{Co}\VV[X,\err(X)].
    \end{align*}
    Assuming that \( \EE[X] = 0 \), which we may do without loss of generality because we make no assumptions on \( a \), \( \textrm{Co}\VV[X,\err(X)] = \EE[X\err(X)] - \EE[X]\EE[\err(X)] = \EE[X\err(X)] \) so
    \begin{align*}
        \VV[\rd(X)] - \VV[X] = \EE[\err(X)^2] - \EE[\err(X)]^2 + 2 \EE[X\err(X)].
    \end{align*}
    Therefore, we obtain a bound on the absolute difference of the variances,
    \begin{align*}
        \abs{ \VV[ \rd(X) ] - \VV[ X ]}
        &\leq 2 \abs{ \EE[ X\err(X) ] } 
        + \max\{ \EE[ \err(X)^2 ] , \EE[ \err(X) ]^2 \}.
    \end{align*}

    Note that we could use that the error function is odd about any point in the mesh to to cancel some contribution to many of the integrals.
This would result in a small improvement in the bounds.

\section{Conclusions}\label{sec5}

In this paper, we provide non-asymptotic analysis for the effects of rounding on the moments of random variables.
Our analysis requires very limited assumptions on the actual structure of the set being rounded to and is therefore applicable to a range of settings.
Moreover, because our bounds are non-asymptotic, they can be used in the parameter ranges for precision encountered in practice.
Our analysis also sheds light on differences between \RtN and \SR.
It is well known that \SR is unbiased, and this property has been used to analyse the scheme in the deterministic setting \cite{connolly_higham_mary_21}.
However, we show that when rounding random variables, this unbiasedness is at the cost of slower convergence of absolute and higher moments.
Indeed, for fixed \( k \), the bounds for \SR are a constant (growing exponentially in \( k \)) factor worse than the bounds for \RtN.
This suggests that in settings where it is important to preserve not just the mean, but also higher moments, \SR may not always be better than \RtN.
In fact, it our analysis opens the possibility for picking randomized rounding schemes based on the relative accuracy constraints for different moments.

\section*{Acknowledgments}

The author thanks Yu-Chen Cheng, Matt Loirg, Tom Togdon, and Ying-Jen Yang for the many helpful discussions and comments.


\subsection*{Financial disclosure}
This work is supported by the National Science Foundation Graduate Research Fellowship Program under Grant No. DGE-1762114.
Any opinions, findings, and conclusions orrecommendations expressed in this material are those of the author and do not necessarily reflect theviews of the National Science Foundation.

\subsection*{Conflict of interest}

The author declares no potential conflict of interests.



\begin{thebibliography}{}

\bibitem [\protect \citeauthoryear {%
Alistarh%
, Grubic%
, Li%
, Tomioka%
\BCBL {}\ \BBA {} Vojnovic%
}{%
Alistarh%
\ \protect \BOthers {.}}{%
{\protect \APACyear {2016}}%
}]{%
alistarh_grubic_li_tomioka_vojnovic_17}
\APACinsertmetastar {%
alistarh_grubic_li_tomioka_vojnovic_17}%
\begin{APACrefauthors}%
Alistarh, D.%
, Grubic, D.%
, Li, J.%
, Tomioka, R.%
\BCBL {}\ \BBA {} Vojnovic, M.%
\end{APACrefauthors}%
\unskip\
\newblock
\APACrefYearMonthDay{2016}{}{}.
\newblock
{\BBOQ}\APACrefatitle {QSGD: Communication-Efficient SGD via Gradient
  Quantization and Encoding} {Qsgd: Communication-efficient sgd via gradient
  quantization and encoding}.{\BBCQ}.
\PrintBackRefs{\CurrentBib}

\bibitem [\protect \citeauthoryear {%
Bai%
, Zheng%
, Zhang%
\BCBL {}\ \BBA {} Hu%
}{%
Bai%
\ \protect \BOthers {.}}{%
{\protect \APACyear {2009}}%
}]{%
bai_zheng_zhang_hu_09}
\APACinsertmetastar {%
bai_zheng_zhang_hu_09}%
\begin{APACrefauthors}%
Bai, Z.%
, Zheng, S.%
, Zhang, B.%
\BCBL {}\ \BBA {} Hu, G.%
\end{APACrefauthors}%
\unskip\
\newblock
\APACrefYearMonthDay{2009}{{\APACmonth{08}}}{}.
\newblock
{\BBOQ}\APACrefatitle {Statistical analysis for rounded data} {Statistical
  analysis for rounded data}.{\BBCQ}
\newblock
\APACjournalVolNumPages{Journal of Statistical Planning and
  Inference}{139}{8}{2526--2542}.
\PrintBackRefs{\CurrentBib}

\bibitem [\protect \citeauthoryear {%
Chaudhuri%
, Motwani%
\BCBL {}\ \BBA {} Narasayya%
}{%
Chaudhuri%
\ \protect \BOthers {.}}{%
{\protect \APACyear {1998}}%
}]{%
chaudhuri_motwani_narasayya_98}
\APACinsertmetastar {%
chaudhuri_motwani_narasayya_98}%
\begin{APACrefauthors}%
Chaudhuri, S.%
, Motwani, R.%
\BCBL {}\ \BBA {} Narasayya, V.%
\end{APACrefauthors}%
\unskip\
\newblock
\APACrefYearMonthDay{1998}{}{}.
\newblock
{\BBOQ}\APACrefatitle {Random Sampling for Histogram Construction: How Much is
  Enough?} {Random sampling for histogram construction: How much is
  enough?}{\BBCQ}
\newblock
\BIn{} \APACrefbtitle {Proceedings of the 1998 ACM SIGMOD International
  Conference on Management of Data.} {Proceedings of the 1998 acm sigmod
  international conference on management of data.}
\newblock
\APACaddressPublisher{New York, NY, USA}{ACM}.
\PrintBackRefs{\CurrentBib}

\bibitem [\protect \citeauthoryear {%
Connolly%
, Higham%
\BCBL {}\ \BBA {} Mary%
}{%
Connolly%
\ \protect \BOthers {.}}{%
{\protect \APACyear {2021}}%
}]{%
connolly_higham_mary_21}
\APACinsertmetastar {%
connolly_higham_mary_21}%
\begin{APACrefauthors}%
Connolly, M\BPBI P.%
, Higham, N\BPBI J.%
\BCBL {}\ \BBA {} Mary, T.%
\end{APACrefauthors}%
\unskip\
\newblock
\APACrefYearMonthDay{2021}{{\APACmonth{01}}}{}.
\newblock
{\BBOQ}\APACrefatitle {Stochastic Rounding and Its Probabilistic Backward Error
  Analysis} {Stochastic rounding and its probabilistic backward error
  analysis}.{\BBCQ}
\newblock
\APACjournalVolNumPages{{SIAM} Journal on Scientific
  Computing}{43}{1}{A566--A585}.
\PrintBackRefs{\CurrentBib}

\bibitem [\protect \citeauthoryear {%
Freedman%
\ \BBA {} Diaconis%
}{%
Freedman%
\ \BBA {} Diaconis%
}{%
{\protect \APACyear {1981}}%
}]{%
freedman_diaconis_81}
\APACinsertmetastar {%
freedman_diaconis_81}%
\begin{APACrefauthors}%
Freedman, D.%
\BCBT {}\ \BBA {} Diaconis, P.%
\end{APACrefauthors}%
\unskip\
\newblock
\APACrefYearMonthDay{1981}{Dec}{01}.
\newblock
{\BBOQ}\APACrefatitle {On the histogram as a density estimator:L2 theory} {On
  the histogram as a density estimator:l2 theory}.{\BBCQ}
\newblock
\APACjournalVolNumPages{Zeitschrift f\:ur Wahrscheinlichkeitstheorie und
  Verwandte Gebiete}{57}{4}{453--476}.
\PrintBackRefs{\CurrentBib}

\bibitem [\protect \citeauthoryear {%
Graf%
\ \BBA {} Luschgy%
}{%
Graf%
\ \BBA {} Luschgy%
}{%
{\protect \APACyear {2007}}%
}]{%
graf_luschgy_07}
\APACinsertmetastar {%
graf_luschgy_07}%
\begin{APACrefauthors}%
Graf, S.%
\BCBT {}\ \BBA {} Luschgy, H.%
\end{APACrefauthors}%
\unskip\
\newblock
\APACrefYear{2007}.
\newblock
\APACrefbtitle {Foundations of quantization for probability distributions}
  {Foundations of quantization for probability distributions}.
\newblock
\APACaddressPublisher{}{Springer}.
\PrintBackRefs{\CurrentBib}

\bibitem [\protect \citeauthoryear {%
Gupta%
, Agrawal%
, Gopalakrishnan%
\BCBL {}\ \BBA {} Narayanan%
}{%
Gupta%
\ \protect \BOthers {.}}{%
{\protect \APACyear {2015}}%
}]{%
gupta_agrawal_gopalakrishnan_narayanan_15}
\APACinsertmetastar {%
gupta_agrawal_gopalakrishnan_narayanan_15}%
\begin{APACrefauthors}%
Gupta, S.%
, Agrawal, A.%
, Gopalakrishnan, K.%
\BCBL {}\ \BBA {} Narayanan, P.%
\end{APACrefauthors}%
\unskip\
\newblock
\APACrefYearMonthDay{2015}{}{}.
\newblock
{\BBOQ}\APACrefatitle {Deep learning with limited numerical precision} {Deep
  learning with limited numerical precision}.{\BBCQ}
\newblock
\BIn{} \APACrefbtitle {International Conference on Machine Learning}
  {International conference on machine learning}\ (\BPGS\ 1737--1746).
\PrintBackRefs{\CurrentBib}

\bibitem [\protect \citeauthoryear {%
Hall%
}{%
Hall%
}{%
{\protect \APACyear {1982}}%
}]{%
hall_82}
\APACinsertmetastar {%
hall_82}%
\begin{APACrefauthors}%
Hall, P.%
\end{APACrefauthors}%
\unskip\
\newblock
\APACrefYearMonthDay{1982}{{\APACmonth{04}}}{}.
\newblock
{\BBOQ}\APACrefatitle {The Influence of Rounding Errors on Some Nonparametric
  Estimators of a Density and its Derivatives} {The influence of rounding
  errors on some nonparametric estimators of a density and its
  derivatives}.{\BBCQ}
\newblock
\APACjournalVolNumPages{{SIAM} Journal on Applied
  Mathematics}{42}{2}{390--399}.
\PrintBackRefs{\CurrentBib}

\bibitem [\protect \citeauthoryear {%
Higham%
}{%
Higham%
}{%
{\protect \APACyear {2002}}%
}]{%
higham_02}
\APACinsertmetastar {%
higham_02}%
\begin{APACrefauthors}%
Higham, N\BPBI J.%
\end{APACrefauthors}%
\unskip\
\newblock
\APACrefYear{2002}.
\newblock
\APACrefbtitle {Accuracy and Stability of Numerical Algorithms} {Accuracy and
  stability of numerical algorithms}\ (\PrintOrdinal{Second}\ \BEd).
\newblock
\APACaddressPublisher{}{Society for Industrial and Applied Mathematics}.
\PrintBackRefs{\CurrentBib}

\bibitem [\protect \citeauthoryear {%
Higham%
\ \BBA {} Mary%
}{%
Higham%
\ \BBA {} Mary%
}{%
{\protect \APACyear {2019}}%
}]{%
higham_mary_19}
\APACinsertmetastar {%
higham_mary_19}%
\begin{APACrefauthors}%
Higham, N\BPBI J.%
\BCBT {}\ \BBA {} Mary, T.%
\end{APACrefauthors}%
\unskip\
\newblock
\APACrefYearMonthDay{2019}{}{}.
\newblock
{\BBOQ}\APACrefatitle {A New Approach to Probabilistic Rounding Error Analysis}
  {A new approach to probabilistic rounding error analysis}.{\BBCQ}
\newblock
\APACjournalVolNumPages{SIAM Journal on Scientific
  Computing}{41}{5}{A2815-A2835}.
\PrintBackRefs{\CurrentBib}

\bibitem [\protect \citeauthoryear {%
Janson%
}{%
Janson%
}{%
{\protect \APACyear {2006}}%
}]{%
janson_06}
\APACinsertmetastar {%
janson_06}%
\begin{APACrefauthors}%
Janson, S.%
\end{APACrefauthors}%
\unskip\
\newblock
\APACrefYearMonthDay{2006}{}{}.
\newblock
{\BBOQ}\APACrefatitle {Rounding of continuous random variables and oscillatory
  asymptotics} {Rounding of continuous random variables and oscillatory
  asymptotics}.{\BBCQ}
\newblock
\APACjournalVolNumPages{The Annals of Probability}{34}{5}{1807--1826}.
\PrintBackRefs{\CurrentBib}

\bibitem [\protect \citeauthoryear {%
Knuth%
}{%
Knuth%
}{%
{\protect \APACyear {2019}}%
}]{%
knuth_19}
\APACinsertmetastar {%
knuth_19}%
\begin{APACrefauthors}%
Knuth, K\BPBI H.%
\end{APACrefauthors}%
\unskip\
\newblock
\APACrefYearMonthDay{2019}{{\APACmonth{12}}}{}.
\newblock
{\BBOQ}\APACrefatitle {Optimal data-based binning for histograms and
  histogram-based probability density models} {Optimal data-based binning for
  histograms and histogram-based probability density models}.{\BBCQ}
\newblock
\APACjournalVolNumPages{Digital Signal Processing}{95}{}{102581}.
\PrintBackRefs{\CurrentBib}

\bibitem [\protect \citeauthoryear {%
Monahan%
}{%
Monahan%
}{%
{\protect \APACyear {1985}}%
}]{%
monahan_85}
\APACinsertmetastar {%
monahan_85}%
\begin{APACrefauthors}%
Monahan, J\BPBI F.%
\end{APACrefauthors}%
\unskip\
\newblock
\APACrefYearMonthDay{1985}{}{}.
\newblock
{\BBOQ}\APACrefatitle {Accuracy in random number generation} {Accuracy in
  random number generation}.{\BBCQ}
\newblock
\APACjournalVolNumPages{Mathematics of Computation}{45}{172}{559--568}.
\PrintBackRefs{\CurrentBib}

\bibitem [\protect \citeauthoryear {%
Schneeweiss%
, Komlos%
\BCBL {}\ \BBA {} Ahmad%
}{%
Schneeweiss%
\ \protect \BOthers {.}}{%
{\protect \APACyear {2010}}%
}]{%
schneeweiss_komlos_ahmad_10}
\APACinsertmetastar {%
schneeweiss_komlos_ahmad_10}%
\begin{APACrefauthors}%
Schneeweiss, H.%
, Komlos, J.%
\BCBL {}\ \BBA {} Ahmad, A.%
\end{APACrefauthors}%
\unskip\
\newblock
\APACrefYearMonthDay{2010}{09}{}.
\newblock
{\BBOQ}\APACrefatitle {Symmetric and asymmetric rounding: A review and some new
  results} {Symmetric and asymmetric rounding: A review and some new
  results}.{\BBCQ}
\newblock
\APACjournalVolNumPages{AStA Advances in Statistical Analysis}{94}{}{247-271}.
\PrintBackRefs{\CurrentBib}

\bibitem [\protect \citeauthoryear {%
Sheppard%
}{%
Sheppard%
}{%
{\protect \APACyear {1897}}%
}]{%
sheppard_97}
\APACinsertmetastar {%
sheppard_97}%
\begin{APACrefauthors}%
Sheppard, W.%
\end{APACrefauthors}%
\unskip\
\newblock
\APACrefYearMonthDay{1897}{11}{}.
\newblock
{\BBOQ}\APACrefatitle {On the Calculation of the most Probable Values of
  Frequency-Constants, for Data arranged according to Equidistant Division of a
  Scale} {On the calculation of the most probable values of
  frequency-constants, for data arranged according to equidistant division of a
  scale}.{\BBCQ}
\newblock
\APACjournalVolNumPages{Proceedings of the London Mathematical
  Society}{s1-29}{1}{353--380}.
\PrintBackRefs{\CurrentBib}

\bibitem [\protect \citeauthoryear {%
Tricker%
}{%
Tricker%
}{%
{\protect \APACyear {1990}}%
}]{%
tricker_90}
\APACinsertmetastar {%
tricker_90}%
\begin{APACrefauthors}%
Tricker, A\BPBI R.%
\end{APACrefauthors}%
\unskip\
\newblock
\APACrefYearMonthDay{1990}{{\APACmonth{01}}}{}.
\newblock
{\BBOQ}\APACrefatitle {The effect of rounding on the significance level of
  certain normal test statistics} {The effect of rounding on the significance
  level of certain normal test statistics}.{\BBCQ}
\newblock
\APACjournalVolNumPages{Journal of Applied Statistics}{17}{1}{31--38}.
\PrintBackRefs{\CurrentBib}

\bibitem [\protect \citeauthoryear {%
Ushakov%
\ \BBA {} Ushakov%
}{%
Ushakov%
\ \BBA {} Ushakov%
}{%
{\protect \APACyear {2017}}%
}]{%
ushakov_ushakov_17}
\APACinsertmetastar {%
ushakov_ushakov_17}%
\begin{APACrefauthors}%
Ushakov, N.%
\BCBT {}\ \BBA {} Ushakov, V.%
\end{APACrefauthors}%
\unskip\
\newblock
\APACrefYearMonthDay{2017}{}{}.
\newblock
{\BBOQ}\APACrefatitle {Statistical analysis of rounded data: Recovering of
  information lost due to rounding} {Statistical analysis of rounded data:
  Recovering of information lost due to rounding}.{\BBCQ}
\newblock
\APACjournalVolNumPages{Journal of the Korean Statistical Society}{46}{3}{426 -
  437}.
\PrintBackRefs{\CurrentBib}

\bibitem [\protect \citeauthoryear {%
{Vardeman}%
}{%
{Vardeman}%
}{%
{\protect \APACyear {2005}}%
}]{%
vardeman_05}
\APACinsertmetastar {%
vardeman_05}%
\begin{APACrefauthors}%
{Vardeman}, S\BPBI B.%
\end{APACrefauthors}%
\unskip\
\newblock
\APACrefYearMonthDay{2005}{Oct}{}.
\newblock
{\BBOQ}\APACrefatitle {Sheppard's correction for variances and the
  "quantization noise model"} {Sheppard's correction for variances and the
  "quantization noise model"}.{\BBCQ}
\newblock
\APACjournalVolNumPages{IEEE Transactions on Instrumentation and
  Measurement}{54}{5}{2117-2119}.
\PrintBackRefs{\CurrentBib}

\bibitem [\protect \citeauthoryear {%
Wilkinson%
}{%
Wilkinson%
}{%
{\protect \APACyear {1963}}%
}]{%
wilkinson_63}
\APACinsertmetastar {%
wilkinson_63}%
\begin{APACrefauthors}%
Wilkinson, J\BPBI H.%
\end{APACrefauthors}%
\unskip\
\newblock
\APACrefYear{1963}.
\newblock
\APACrefbtitle {Rounding errors in algebraic processes} {Rounding errors in
  algebraic processes}.
\newblock
\APACaddressPublisher{}{Prentence Hall Inc.}
\PrintBackRefs{\CurrentBib}

\bibitem [\protect \citeauthoryear {%
Wilrich%
}{%
Wilrich%
}{%
{\protect \APACyear {2005}}%
}]{%
wilrich_05}
\APACinsertmetastar {%
wilrich_05}%
\begin{APACrefauthors}%
Wilrich, P\BHBI T.%
\end{APACrefauthors}%
\unskip\
\newblock
\APACrefYearMonthDay{2005}{{\APACmonth{01}}}{}.
\newblock
{\BBOQ}\APACrefatitle {Rounding of measurement values or derived values}
  {Rounding of measurement values or derived values}.{\BBCQ}
\newblock
\APACjournalVolNumPages{Measurement}{37}{1}{21--30}.
\PrintBackRefs{\CurrentBib}

\end{thebibliography}

\appendix
These sections do not appear in the peer-reviewed journal version.

\section{Asymptotic bounds}

While \cref{thm:err_int_abs_unif} requires a uniform spacing of points, we expect a similar result to hold if the points are locally uniformly spaced.
In order to make this intuition precise, we introduce a definition based on the study of the fine structure of the zeros of orthogonal polynomials.
\begin{definition}
    \label{def:unif_clock}
    Let \( \nu : [-1,1] \to \R \) be a continuous, non-vanishing, probability density function.
    We say a sequence of sets of points \( \{ \{ p_{n,i} \}_{i=1}^{n} \}_{n=1}^{\infty} \) has uniform clock behavior with respect to \( \nu \) if
    \begin{align*}
        \lim_{n\to\infty} \sup_{i<n} \{ | n ( p_{n,i+1} - p_{n,i} ) - \nu(p_{n,i})^{-1} | \} = 0.
    \end{align*}
\end{definition}

\begin{theorem}
    \label{thm:asymptotic}
    Suppose that \( \{ \{ p_{j,i} \}_{i=1}^{j} \}_{j=1}^{\infty} \) has uniform clock behavior with respect to \( \nu \).
    Let \( \rd_n : \R \to \{ p_{n,i} \}_{i=1}^{n} \) denote the rounding function to the \( n \)-th set and \( \err_n:\R\to\R \) the corresponding error function.
    Then, for any \( [a,b] \subseteq [-1,1] \), as \( n\to\infty \), with \( e_{\rd} \) as in \cref{table:constants},
    \begin{enumerate}[label=(\roman*),resume]
        \item 
            if \( \rd:\R\to\mathbb{F} \) is \RtN,
            \begin{align*}  
                \abs[\bigg]{ \int_{a}^{b}  \abs{ \err_n(x) }^k \d{x} 
                - \left[ 2^{-k} e_{\rd}(k) \int_{a}^{b} \nu(x)^{-k} \d{x}  \right] \cdot n^{-k} } =  o(n^{-k}).
        \end{align*}
        \item 
            if \( \rd:\R\to\mathbb{F} \) is \SR, 
             \begin{align*}  
                 \abs[\bigg]{ \int_{a}^{b}  \EE_{\rd}\!\left[ \abs{ \err_n(x) }^k \right] \d{x} 
                 - \left[ e_{\rd}(k) \int_{a}^{b} \nu(x)^{-k} \d{x}  \right] \cdot n^{-k} } =  o(n^{-k}).
            \end{align*}
    \end{enumerate}
\end{theorem}

\begin{proof}
We first prove the \RtN case.
Let \( j = j(n) \) such that \( p_{n,j-1} < a < p_{n,j} \) and  \( j' = j'(n) \) such that \( p_{n,j'} < a < p_{n,j'+1} \).
Then, by direct computation,
\begin{align*}
    \int_{p_{n,j}}^{p_{n,j'}}  \abs{ \err_n(x) }^k \d{x} 
    &= \sum_{i=j}^{j'} \int_{p_{n,i}}^{p_{n,i+1}} \abs{\err_n(x)}^k \d{x}
    = \sum_{i=j}^{j'} 2^{-k} \frac{(p_{n,i+1} - p_{n,i} )^{k+1} }{k+1}.
\end{align*}
Clearly \( |a - p_{n,j}| = O(n^{-1}) \) and \( |b-p_{n,j'}| = O(n^{-1}) \) so
\begin{align*}
    \int_{a}^{b}  \abs{ \err_n(x) }^k \d{x} 
    = \lim_{n\to\infty} \sum_{i=j}^{j'} 2^{-k} \frac{(p_{n,i+1} - p_{n,i} )^{k+1} }{k+1}.
\end{align*}
Similarly, writing the left Reimann integral of \( \nu^{-k} \),
\begin{align*}
    \int_{a}^{b} \nu(x)^{-k} \d{x} 
    &= \lim_{n\to\infty} \sum_{i=j}^{j'} \nu(p_{n,i})^{-k} (p_{n,i+1} - p_{n,i}) 
    = \lim_{n\to\infty} \sum_{i=j}^{j'} (n(p_{n,i}-p_{n,i+1}))^k (p_{n,i+1} - p_{n,i}).
    \tag*{\qedhere}
\end{align*}
    
\noindent
For the \SR case we note that
\begin{align*}
    \int_{p_{n,j}}^{p_{n,j'}}  \EE_{\rd}[ \abs{ \err_n(x) }^k ] \d{x} 
    &= \sum_{i=j}^{j'} \int_{p_{n,i}}^{p_{n,i+1}} \abs{\err_n(x)}^k \d{x}
    = \sum_{i=j}^{j'} \frac{2}{k^2+3k+2} (p_{n,i+1} - p_{n,i} )^{k+1}.
\end{align*}
The rest of the proof follows in the identical 
\end{proof}

\section{Example (asymptotic bounds)}
In this example, we illustrate the asymptotic bounds from \cref{thm:asymptotic}. 

Recall that a Beta random variable with parameters \( \alpha,\beta > 0 \) has density function given by
\begin{align*}
    x\mapsto \frac{\Gamma(\alpha + \beta)}{\Gamma(\alpha)\Gamma(\beta)}x^{\alpha-1} (1-x)^{\beta-1}.
\end{align*}
We let \( Y \sim \operatorname{Beta}(\alpha,\beta) \) and take \( X = 2Y-1 \) (so that \( X \) is supported on \( [-1,1] \)).
    
We round this random variable to the roots of the Chebyshev polynomials of the first kind 
\begin{align*}
    p_{n,i} = \cos \left( \frac{2i-1}{n}\pi \right)
    ,&& i=1,\ldots, n
\end{align*}
which have uniform clock behavior with respect to \( \nu(x) = 1/(\pi\sqrt{(1+x)(1-x)}) \) on \( (-1,1) \).

We seek to compute the absolute moments \( \EE[\abs{\err_n(x)}^k] \) for \RtN and \( \EE[\EE_{\rd}[\abs{\err_n(x)}^k]] \) for \SR.
Respectively, these are directly computed by the integrals
\begin{align*}
    \int_{a}^{b}  f_X(x) \abs{ \err_n(x) }^k \d{x} 
    ,&&
    \int_{a}^{b} f_X(x) \EE_{\rd} \!\left[ \abs{ \err_n(x) }^k \right] \d{x}.
\end{align*}
By \cref{thm:asymptotic} and \cref{thm:prod_bound} we have asympototic approximations for these integrals
\begin{align*}
    \left[ 2^{-k} c_{\rd}(k) \int_{a}^{b} f_X(x) \nu(x)^{-k} \d{x}  \right] \cdot n^{-k}
    ,&&
    \left[ c_{\rd}(k) \int_{a}^{b} f_X(x) \nu(x)^{-k} \d{x}  \right] \cdot n^{-k}.
\end{align*}

All four quantities are show in \cref{fig:asymptotic} for \( k=1,2,3 \) and parameters \( \alpha = 5 \) and \( \beta = 10 \).
In particular, note that the convergence of \RtN is a constant (growing exponentialy in \( k \)) better than that of \SR.

\begin{figure*}[t]\centering
\includegraphics[width=\textwidth]{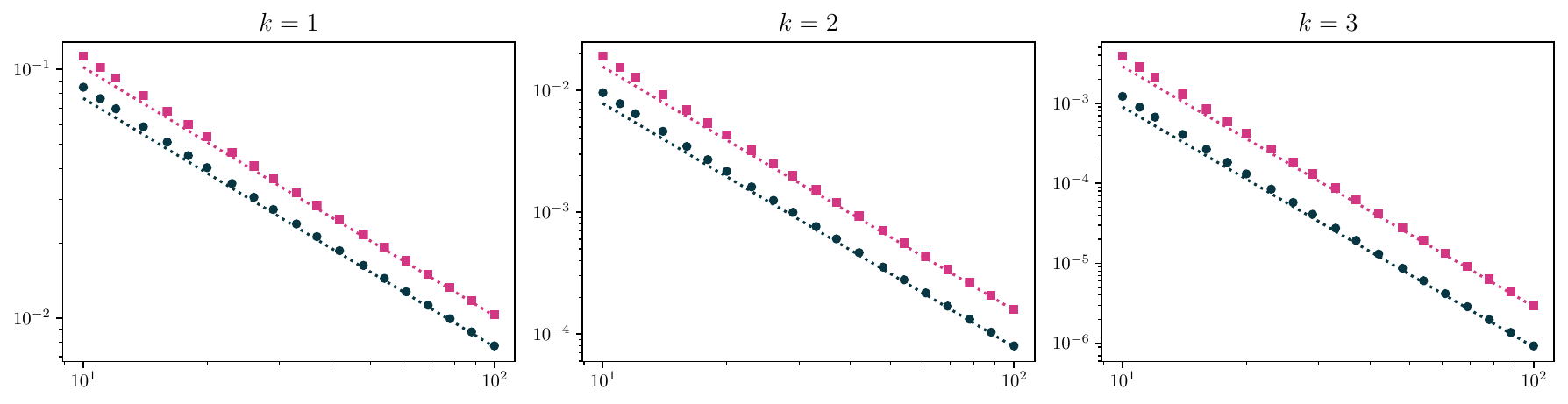}
\caption{
Note that the vertical range of the three plots is significantly different, since the asymptotic rate of convergence is \( n^{-k} \).
\emph{Legend}:
true integral RtN ({\protect\raisebox{-.2mm}{\protect\includegraphics[scale=1]{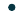}}}),
true integral SR ({\protect\raisebox{-.2mm}{\protect\includegraphics[scale=1]{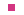}}}),
asymptotic RtN ({\protect\raisebox{-.2mm}{\protect\includegraphics[scale=1]{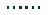}}}),
asymptotic SR ({\protect\raisebox{-.2mm}{\protect\includegraphics[scale=1]{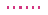}}}).
}
\label{fig:asymptotic}
\end{figure*}

\section{Additional details for selected proofs}

In this section we provide some additional details for proofs.

\begin{proof}[Proof of \cref{thm:stoch_err_ex}]
    Note that if \( \rd(x) = \flx \) then \( \err(x) = \flx - x \) whereas if \( \rd(x) = \clx \) then \( \err(x) = \clx - x \).
    Thus,
\begin{align*}
        \EE_{\rd}\!\left[ \err(x)^k \right]
        &= (\flx - x)^k \PP[\rd(x) = \flx] + (\clx-x)^k \PP[\rd(x) = \clx]
        \\&= (\flx - x )^k \left( 1- \frac{x-\flx}{\clx - \flx} \right)
        + (\clx - x )^k \left( \frac{x-\flx}{\clx - \flx} \right) 
        \\
        \EE_{\rd}\!\left[  \abs{ \err(x) } ^k \right]
        &= |\flx - x|^k \PP[\rd(x) = \flx] + |\clx-x|^k \PP[\rd(x) = \clx]
        \\&=  ( x - \flx)^k \left( 1- \frac{x-\flx}{\clx - \flx} \right)
        +  (\clx - x)^k \left( \frac{x-\flx}{\clx - \flx} \right).
        \tag*{\qedhere}
\end{align*}
\end{proof}

\begin{proof}[Proof of \cref{thm:err_endpoint} (detailed)]
    For \RtN \( c' = \rd(c) \). Note that \( \err(x) = \rd(c) - x \) does not change signs on \( [\min(c,c'),\max(c,c')] \).
    Thus
    \begin{align*}
        \abs[\Bigg]{ \int_{c}^{c'} (\rd(c)-x)^k \d{x} }
        = \int_{\min(c,c')}^{\max(c,c')} |c'-x|^k \d{x} 
    \end{align*}
    Next, since \( \err(x) \) is symmetric about \( \bar{c} = (\flc+\clc)/2 \) over the interval \( [\flc,\clc] \), we can, without loss of generality, assume that \( c' = \flx \) and adjust \( c \) accordingly so that the distance to \( \flc \) is the same as the previous distance to \( c' \).
    By direct computation we have that 
    \begin{align*}
        \int_{\flc}^{c} |c'-x|^k \d{x}
        \int_{\flc}^{c} |\flc-x|^k \d{x}
        = \frac{1}{k+1} (c-\flc)^{k+1}
    \end{align*}
    Thus, since \(  \abs{c-c'} \leq \epsilon \: E(c)  \),
      \begin{align*}
        \abs[\Bigg]{ \int_{c}^{c'} (\rd(c)-x)^k \d{x} }
        \leq \left[ \frac{ E(c) ^{k+1}}{k+1} \right] \cdot \epsilon^{k+1}.
    \end{align*}
    
    \noindent
    For \SR, \( \EE_{\rd}[\err(x)^k ] \) is also symmetric about \( \bar{c} \) over \( [\flc,\clc] \) and does not change signs on \( [\min(c,c'),\max(c,c')] \).
    Thus, akin to the previous case, and again using direct computation of the final integral,
    \begin{align*}
        \abs[\Bigg]{ \int_{c}^{c'} \EE[\err(x)^k] \d{x} }
        = \int_{c}^{c'} \abs[\big]{\! \EE[\err(x)^k] } \d{x} 
        \leq \int_{\flc}^{\bar{c}} \abs[\big]{\! \EE[\err(x)^k] } \d{x} 
        = \frac{1-(k+3)2^{-(k+1)}}{k^2+3k+2} (\clc-\flc)^{k+1}
    \end{align*}
    For \SR we have that \( \clc-\flc \leq \epsilon \: E(c) \) so
    \begin{align*}
        \abs[\Bigg]{ \int_{c}^{c'} \EE[\err(x)^k] \d{x} }
        \leq \left[ \frac{1-(k+3)2^{-(k+1)}}{k^2+3k+2} E(c)^{k+1} \right] \cdot \epsilon^{k+1}.
        \tag*{\qedhere}
    \end{align*}
\end{proof}

\begin{proof}[Proof of \cref{thm:err_int_abs_fp} (detailed)]
    We first prove the \RtN case.
    By direct computation, 
    \begin{align*}
        \int_{\flc}^{\clc} \abs{\err(x)}^k \d{x}
        &= \frac{1}{k+1}\int_{\flc}^{\clc} \abs[\bigg]{ \err \left( \frac{\flc+\clc}{2} \right)  }^k \d{x}.
    \end{align*}
    Next note that since \( E(x) \) is linear over the interval, the tangent like to \( \epsilon^k E(x)^k \) at \( \bar{c} = (\flc+\clc)/2 \) lies entirely below \( \epsilon^k E(x)^k \).
    Moreover, the integral of this tangent like over the interval \( [\flc,\clc] \) is equal to the integral of the the constant function \( \epsilon^k E( \bar{c})^k \). 
    Therefore,
    \begin{align*}
        \frac{1}{k+1}\int_{\flc}^{\clc} \abs[\bigg]{ \err \left( \frac{\flc+\clc}{2} \right)  }^k \d{x}.
        \leq \left[ \frac{1}{k+1} \int_{\flc}^{\clc} E \left( \frac{\flc+\clc}{2} \right)^k \d{x} \right] \cdot \epsilon^k
        \leq \left[ \frac{1}{k+1} \int_{\flc}^{\clc} E(x)^k \d{x} \right] \cdot \epsilon^k.
    \end{align*}
    Combining these inequalities and using this and the fact that \( \err((\flc+\clc)/2) \leq \epsilon \: E((\flc+\clc)/2) \) we then find,
    \begin{align*}
        \int_{\flc}^{\clc} \abs{\err(x)}^k \d{x}
        \leq \left[ \frac{1}{k+1} \int_{\flc}^{\clc} E(x)^k \d{x} \right] \cdot \epsilon^k.
    \end{align*}

    \noindent
    For the \SR case, by direct computation we have
    \begin{align*}
        \int_{\flc}^{\clc} \EE_{\rd}\!\big[|\err(x)|^k\big] \d{x}
        = \frac{2}{k^2+3k+2} \left( \clc-\flc \right)^{k+1}
        = \frac{2}{k^2+3k+2} \int_{\flc}^{\clc} \left( \clc-\flc \right)^{k} \d{x}.
    \end{align*} 
    Thus, using the fact that \( (\flc-\clc) \leq E(x) \), we find
    \begin{align*}
        \int_{\flc}^{\clc} \EE_{\rd}\!\big[|\err(x)|^k\big] \d{x}
        \leq \left[ \frac{2}{k^2+3k+2} \int_{\flc}^{\clc} E(x)^k \d{x} \right] \cdot \epsilon^k
   \tag*{\qedhere}
    \end{align*}
\end{proof}



\end{document}